\documentclass[10pt]{amsart}
\usepackage{latexsym, amsmath,amssymb}

\renewcommand{\epsilon}{\varepsilon}
\newcommand{\newsection}[1]
{\subsection{#1}\setcounter{theorem}{0} \setcounter{equation}{0}
\par\noindent}

\newtheorem{theorem}{Theorem}

\newtheorem{lemma}[theorem]{Lemma}
\newtheorem{corr}[theorem]{Corollary}

\newtheorem{proposition}[theorem]{Proposition}
\newtheorem{deff}[theorem]{Definition}

\newcommand{\bth}{\begin{theorem}}
\newcommand{\ble}{\begin{lemma}}
\newcommand{\bcor}{\begin{corr}}

\newcommand{\bdeff}{\begin{deff}}

\newcommand{\bprop}{\begin{proposition}}
\newcommand{\ele}{\end{lemma}}
\newcommand{\ecor}{\end{corr}}
\newcommand{\edeff}{\end{deff}}

\newcommand{\eprop}{\end{proposition}}

\newcommand{\Rt}{{\mathbb R}^2}

\newcommand{\la}{\lambda}

\newcommand{\e}{\varepsilon}
\renewcommand{\l}{\lambda}

\newcommand{\supp}{\text{supp }}
\renewcommand{\Pi}{\varPi}

\renewcommand{\epsilon}{\varepsilon}

\newcommand{\R}{{\mathbb R}}

\newcommand{\sqrtd}{\sqrt{-\Delta_g}}

\newcommand{\dgt}{d_{\tilde g}}
\newcommand{\tg}{\tilde g}

\newcommand{\sqrtg}{\sqrt{\Delta_{\tilde g}}}

\newcommand{\sqdt}{\sqrt{-\Delta_{\tilde g}}}
\newcommand{\Rtwo}{{\mathbb R}^2}

\newcommand{\pg}{\phi_{\gamma,\alpha}}

\begin{document}

%	\subjclass[2010]{Primary, 35F99; Secondary 35L20, 42C99}
%	\keywords{Eigenfunctions, negative curvature}

\title
%{A note on geodesic period integrals of eigenfunctions}
{On  integrals of eigenfunctions over geodesics}

	\thanks{The authors were supported in part by the NSF grant DMS-1069175 and the Simons Foundation.}

	\author{Xuehua Chen}
	\author{Christopher D. Sogge}
	\address{Department of Mathematics,  Johns Hopkins University,
	Baltimore, MD 21218}

\begin{abstract}  If $(M,g)$ is a compact Riemannian surface then the integrals of $L^2(M)$-normalized eigenfunctions $e_j$ over  geodesic segments of fixed
length are uniformly bounded.
Also, if $(M,g)$ has negative curvature  and $\gamma(t)$ is a geodesic parameterized by arc length, the measures $e_j(\gamma(t))\, dt$ on $\R$
tend to zero in the sense of distributions
 as the eigenvalue $\la_j\to \infty$, and so integrals of eigenfunctions over periodic geodesics tend to zero as $\la_j\to \infty$.  The assumption of negative curvature is necessary for the latter result.
\end{abstract}

\maketitle

\newsection{General results}

If $M$ is a compact hyperbolic surface   Good~\cite{Gd} and Hejhal~\cite{He}, using Kuznecov formulae,
showed that if $\gamma$ is a periodic geodesic and $ds$ is the associated arc length measure then
\begin{equation}\label{1}
\Bigl|\int_\gamma e_\la \, ds\Bigr|\le C_\gamma,
\end{equation}
with $e_\la$ denoting the $L^2$-normalized eigenfunctions on $M$, i.e., $-\Delta_g e_\la=\la^2e_\la$, and $\|e_\la\|_{L^2(M)}=1$.

This result was  generalized by Zelditch~\cite{ZK} who showed, among many other things, that given any $n$-dimensional compact Riemannian manifold, one has
uniform bounds for integrals of eigenfunctions over closed hypersurfaces.\footnote{We are grateful to Steve Zelditch for helpful comments and referring us to the work based on the Kuznecov formula.} 
Moreover, if $\la_j$ are the eigenvalues of $\sqrt{-\Delta_g}$ and $a_j(\gamma)$ denotes the integral in \eqref{1} with
$\la=\l_j$, then \cite[Lemma 3.1]{ZK} says that $\sum_{\la_j\le \la}|a_j(\gamma)|^2=c_\gamma \,  \la +O(1)$, which implies \eqref{1}.  Note that since,
by the Weyl law $\#\{j: \, \la_j\le \la\}\approx \la^2$, Zelditch's formula says that most of the $a_j(\gamma)$ are much smaller than $1$.

Reznikov~\cite{rez} discussed this problem  and, moreover, initiated work on the related problem of obtaining restriction estimates for geodesics.  The sharp $L^2$ estimates for general Riemannian surfaces were obtained by Burq, G\'erard and Tzvetkov~\cite{burq}.  If $\varPi$ denotes the space of 
unit length geodesics in a two-dimensional compact Riemannian manifold, then one of the results in \cite{burq} is that we have bounds
of the form
\begin{equation}\label{2}
\Bigl(\int_\gamma |e_\la|^2 \, ds\Bigr)^{1/2}\lesssim \la^{\frac14}\|e_\la\|_{L^2(M)}, \quad \gamma \in \varPi.
\end{equation}
This estimate is sharp since it is saturated by the highest weight spherical harmonics on $S^2$.  Recently, improvements under the assumption of nonpositive curvature have been obtained by Sogge and Zelditch~\cite{SZ4} and Chen and Sogge~\cite{ChSo}.  Work showing how these restriction estimates are related to $L^p(M)$ estimates for eigenfunctions is in Bourgain~\cite{bourgainef} and Sogge~\cite{Sokakeya}.

Returning to \eqref{1}, note that the estimate cannot be improved when the compact hyperbolic surface $M$ is replaced by the standard
two-sphere  $S^2$ or two-torus
${\mathbb T}^2$.  For on $S^2$ zonal functions of even order saturate the bound, while for every
periodic geodesic on ${\mathbb T}^2$ one can find a sequence of eigenvalues $\la_j\to\infty$ and corresponding $L^2$-normalized eigenfunctions $e_{\la_j}$ having
the property that $e_{\la_j}\equiv 1$ on $\gamma$.

We shall %generalize the results of \cite{Pitt} and \cite{rez} and 
start by giving a quick proof %of slight generalization 
of a result in \cite{ZK}
saying that we have the analog of \eqref{1} for all geodesic segments in  any Riemannian surface.  The proof will serve as a template for the improvements in the next section of the bounds in \eqref{1} for Riemannian surfaces of negative curvature which appear to be new.

\begin{theorem}\label{theorem1}  Let $(M,g)$ be a two-dimensional compact Riemannian manifold.  Then there is a constant $C=C(M,g)$ so that
\begin{equation}\label{3}
\Bigl|\int_\gamma e_\la \, ds\Bigr|\le C\|e_\la\|_{L^2(M)}, \quad \gamma \in \varPi.
\end{equation}
\end{theorem}

\begin{proof}  Fix an even function $\rho\in {\mathcal S}(\R)$ satisfying $\rho(0)=1$ and $\Hat \rho(t)=0$, $|t|\ge 1/4$, assuming, as we may, that
the 
injectivity radius of $(M,g)$ is ten or more.  Then since $\rho(\la-\sqrtd))e_\la =e_\la$, in order to prove \eqref{2}, it suffices to show that
\begin{equation}\label{4}
\Bigl|\int_\gamma \rho(\la-\sqrtd)f\, ds\Bigr|\le C\|f\|_{L^2(M)}, \quad \gamma \in \varPi.
\end{equation}
Let $\gamma(t)$, $0\le t\le 1$, be a parameterization of $\gamma$ by arc length and
$$\rho(\la-\sqrtd)(x,y)=\sum \rho(\la-\la_j)e_j(x)\overline{e_j(y)}$$
denote the kernel of the operator in \eqref{4}.  Here, $\{e_j\}$ is an orthonormal basis of eigenfunctions with eigenvalues $\{\la_j\}$.

By Schwarz's inequality, we would have \eqref{4} if we could show that
$$\int_M\Bigl|\int_0^1 \sum_j  \rho(\la-\la_j)e_j(\gamma(t))\overline{e_j(y)}\, dt\Bigr|^2 \, dV_g(y)\le C.$$
By orthogonality, if $\chi(\tau)=(\rho(\tau))^2$, this is equivalent to showing that
\begin{equation}\label{5}
\left| \int_0^1\int_0^1 \sum_j \chi(\la-\la_j)e_j(\gamma(t)) \, \overline{e_j(\gamma(s))} \, dt ds\right|\le C.
\end{equation}
Next, we note that the proof of \cite[Lemma 5.1.3]{S3} shows that if $d_g$ denotes the Riemannian distance then we can write
\begin{equation}\label{6}
\sum_j \chi(\la-\la_j)e_j(x) \, \overline{e_j(y)}=\la^{\frac12}\sum_\pm a_\pm(\la;d_g(x,y))e^{\pm i \la d_g(x,y)} +O(1),
\end{equation}
where for every fixed $j=0,1,2,\dots$ we have the uniform bounds
\begin{equation}\label{7}
\Bigl|\frac{d^j}{dr^j}a_\pm(\la;r)\Bigr|\le C_jr^{-j-\frac12}, \quad \text{if } \, \, r\ge \la^{-1},
\end{equation}
and
\begin{equation}\label{8}
|a_\pm(\la,r)|\le \la^{\frac12}, \quad \text{if } \, \, r\in [0,\la^{-1}].
\end{equation}
To obtain \eqref{6}-\eqref{8}, as in \cite[\S 5.1]{S3} one uses H\"ormander's parametrix for the half-wave operators 
$e^{it\sqrtd}$, as well as the fact that $\Hat \chi(t)=0$ for $|t|>1$ and our assumption that the injectivity radius of $(M,g)$ is ten or more.

Since $d_g(\gamma(t),\gamma(s))=|t-s|$, we conclude that we would have \eqref{5} if
\begin{equation}\label{9}
\la^{\frac12}\Bigl|\int_0^1\int_0^1 e^{\pm i\la|t-s|}a_\pm(\la;|t-s|)\, dtds\Bigr|\le C.
\end{equation}
Since this is a trivial consequence of \eqref{7} and \eqref{8}, the proof is complete.
\end{proof}

Reznikov in \cite{rez} and \cite{rez2} also discusses the problem of integrals of eigenfunctions over geodesics circles in compact hyperbolic surfaces and obtains the analog of \eqref{1} for them.  More general results were obtained earlier by Zelditch~\cite[Corollary 3.3]{ZK}, and the proof of Theorem~\ref{theorem1} can also
be used to obtain these special cases of the latter: 
%The proof of Theorem~\ref{theorem1} also yields the following more general result %which is also essentially 
%in \cite{ZK}.

\begin{theorem}\label{theorem2}  Let $(M,g)$ be a compact two-dimensional Riemannian manifold.  Then if $\sigma$ is a unit-length curve in $M$ there
is a constant $C_\sigma$ so that
\begin{equation}\label{10}
\Bigl| \int_\sigma e_\la \, ds\Bigr|\le C_\sigma \|e_\la\|_{L^2(M)}.
\end{equation}
Similar bounds with a uniform constant hold for small smooth perturbations of $\sigma$.
\end{theorem}

\newsection{Improved results for negative curvature}

We conclude our note by showing that we can improve the bounds in \eqref{1}  if we assume that the curvature of $(M,g)$ is strictly negative.  As noted before, this assumption is necessary since the corresponding result is false for the two-sphere and the two-torus.  We shall now assume that
$\gamma(t)$, $t\in \R$, is a geodesic in $\R$ parameterized by arc length, and our main result is the following

\begin{theorem}\label{theoremneg}  Let $(M,g)$ be a  negatively curved compact 2-dimensional Riemannian manifold.  Then the 
measures $\{e_j(\gamma(t)) \, dt\}$ on $\R$ go to zero in the sense of distributions, by which we mean
that if $b\in C^\infty_0(\R)$ then
\begin{equation}\label{2.1}
\int b(t) \, e_j(\gamma(t)) \, dt \to 0, \quad \text{as } \, j\to \infty.
\end{equation}
Consequently, if $\gamma_{per}$ is a periodic geodesic of minimal period $\ell >0$, we have
$$\int_0^\ell e_j(\gamma_{per}(t))\, dt \to 0, \quad \text{as } \, \, j\to \infty.$$
%\begin{equation}\label{2.1}
%\limsup_{\la_j\to \infty}\int_\gamma e_{\la_j} \, ds =0, \quad \text{if } \, \, \gamma \in \varPi.
%\end{equation}
%As a result, we have the analog of \eqref{2.1} for every periodic geodesic $\gamma$ if the integral is taken over a single period.
\end{theorem}

The second part of the lemma follows from the first part via a partition of unity argument.
The proof of \eqref{2.1} shares some similarities with the related $L^p(\gamma)$ restriction estimates for eigenfunctions of Sogge and Zelditch~\cite{SZ4} and Chen and Sogge~\cite{ChSo}.  In particular, the oscillatory integral arguments and simple geometric facts that we shall employ are very similar to those in \cite{ChSo}.

To prove \eqref{2.1} we may assume that the injectivity radius of $(M,g)$ is ten or more and that
$$\text{supp }b\subset [-\tfrac12,\tfrac12].$$
Next, we notice that if, as before, $\rho\in {\mathcal S}(\R)$ is even and satisfies $\rho(0)=1$ and $\Hat \rho(t)=0$, $|t|\ge 1/4$, then given
$T\gg 1$ we have $\rho(T(\la-\sqrtd)))e_\la =e_\la$.  As a result, in order to prove \eqref{2.1} it suffices to verify that if $T\gg 1$ then
\begin{equation}\label{2.2}
\Bigl|\int b(s) \bigl(\rho(T(\la-\sqrtd))f\bigr)(\gamma(s))\, ds\Bigr|\le \bigl(CT^{-\frac12}+C_T\la^{-\frac14}\bigr)\|f\|_{L^2(M)},
\end{equation}
where $C$ is independent of $T\gg1$, but not $C_T$.  Repeating the argument which showed how \eqref{5} implies \eqref{4}, we conclude
that if $\chi(\tau)=(\rho(\tau))^2$ and 
$$b(t,s)=b(t)b(s),$$
 then we would have \eqref{2.2} if we could show that
\begin{equation}\label{2.3}
\Bigl|\iint b(t,s)\sum_j \chi(T(\la-\la_j))e_j(\gamma(t))\overline{e_j(\gamma(s))}\, dt ds\Bigr|\le CT^{-1}+C_T\la^{-\frac12}.
\end{equation}
%In proving this, as before, we may assume that the injectivity radius of $(M,g)$ is ten or more.

Note that
$$\sum_j \chi(T(\la-\la_j))e_j(x)\overline{e_j(y)}=\frac1{2\pi T}\int \Hat \chi(\tau/T) e^{-i\tau\la} e^{i\tau \sqrtd}\, d\tau.$$
If we pick a bump function $\beta\in C^\infty_0(\R)$ satisfying
$$\beta(\tau)=1, \quad |\tau|\le 3, \, \, \, \text{and } \, \beta(\tau)=0, \, \, |\tau|\ge 4,$$
the proof of \eqref{5} shows that if $\Psi_T$ denotes the inverse Fourier transform of
$\tau\to \beta(\tau)\chi(\tau/T)$, then
\begin{multline*}\Bigl|\frac1{2\pi T}\iint\int b(t,s) \, \beta(\tau)\Hat  \chi(\tau/T) e^{-i\tau\la} \bigl(e^{i\tau \sqrtd}\bigr)(\gamma(t),\gamma(s))\, d\tau dt ds\Bigr|
\\
=T^{-1}\Bigl| \iint b(t,s) \sum_j \Psi_T(\la-\la_j)e_j(\gamma(t))\overline{e_j(\gamma(s))}\, dt ds\Bigr|
\le CT^{-1}.\end{multline*}
Here,
$ \bigl(e^{i\tau \sqrtd}\bigr)(x,y)=\sum_j e^{i\tau\la_j}e_j(x)\overline{e_j(y)}$
denotes the kernel of the half-wave operator $e^{i\tau \sqrtd}$.

Based on the preceding inequality, in order to prove \eqref{2.3}, it suffices to show that
\begin{multline}\label{2.4}
\Bigl|\frac1{2\pi }\iiint b(t,s)\,  (1-\beta(\tau))\Hat  \chi(\tau/T) e^{-i\tau\la} \bigl(e^{i\tau \sqrtd}\bigr)(\gamma(t),\gamma(s))
\, d\tau dt ds\Bigr|
\\
\le 1+C_T\la^{-\frac12}.\end{multline}
We shall need to use the fact that since $\Hat \chi =(2\pi)^{-1}\Hat \rho * \Hat \rho$, we have
\begin{equation}\label{2.5} \Hat \chi(\tau)=0, \quad |\tau|\ge \frac12,
\end{equation}
which means that the $\tau$ integrand in the left side of \eqref{2.4} vanishes when $|\tau|\ge T/4$.
We can make one more easy reduction.  If $\Phi_T$ denotes the inverse Fourier transform of
$\tau \to  (1-\beta(\tau))\chi(\tau/T)$ then $\Phi_T\in {\mathcal S}(\R)$ and consequently
$$\sum_j \Phi_T(\la+\la_j) \, e_j(\gamma(t))\overline{e_j(\gamma(s))}=O_{T,N}((1+\la)^{-N})$$
for any $N=1,2,3,\dots$.  Thus, by Euler's formula and \eqref{2.5}, in order to prove \eqref{2.4}, it suffices to show that
\begin{multline}\label{2.6}
\Bigl|\iint\int_{-T/2}^{T/2} b(t,s) \, (1-\beta(\tau))\Hat  \chi(\tau/T) e^{-i\tau\la} \bigl(\cos \tau\sqrtd \bigr)(\gamma(t),\gamma(s))
\, d\tau dt ds\Bigr|
\\
\le 1+C_T\la^{-\frac12}.\end{multline}
Here $(\cos \tau \sqrtd)(x,y)$ is the kernel for the map $C^\infty(M)\ni f\to u\in C^\infty(\R\times M)$, where $u(t,x)$ is the solution of the Cauchy problem
with initial data $(f,0)$, i.e.
\begin{equation}\label{2.7}(\partial_t^2-\Delta_g)u=0, \, \, \, u(0,\, \cdot\, )=f, \, \, \, \partial_t u(0, \, \cdot \, )=0.
\end{equation}

To be able to compute the integral in \eqref{2.6} we need to relate this wave kernel to the corresponding one in the universal cover for $(M,g)$.  Recall that by a theorem of Hadamard (see \cite[Chapter 7]{do Carmo}) for every point $P\in M$, the exponential map at $P$, $\exp_P: T_PM\to M$ is a covering map.
We might as well take $P=\gamma(0)$ to be the midpoint of the 
geodesic segment$\{\gamma(t): \, |t|\le \tfrac12\}$.  If we identify $T_PM$ with $\Rt$, and let $\kappa$ denote this exponential
map then $\kappa: \Rt\to M$ is a covering map.  We also will denote by $\tilde g$ the metric on $\Rt$ which is the pullback via $\kappa$ of the the metric
$g$ on $M$.  Also, let $\Gamma$ denote the group of deck transformations, which are the diffeomorphisms $\alpha$ from $\Rt$ to itself preserving $\kappa$, i.e., $\kappa = \kappa \circ \alpha$.  
Next, let
$$D_{Dir}=\{ \tilde y\in \Rt: \, d_{\tg}(0,\tilde y)<\dgt(0, \alpha(\tilde y)), \, \forall \alpha \in \Gamma, \, \, \alpha \ne Identity\}$$
be the Dirichlet domain for $(\Rt,\tg)$, where $\dgt(\, \cdot \, ,\, \cdot \, )$ denotes the Riemannian distance function for $\Rt$ corresponding to
the metric $\tg$.   We can then add to $D_{Dir}$ a subset of $\partial D_{Dir}=\overline{D_{Dir}}\backslash
\text{Int }(D_{Dir})$ to obtain a natural fundamental domain $D$, which has the property that $\Rt$ is the disjoint union of the
$\alpha(D)$ as $\alpha$ ranges over $\Gamma$ and $\{\tilde y\in \Rt: \, \dgt(0,\tilde y)<10\} \subset D$ since we are assuming that the injectivity
radius of $(M,g)$ is more than ten.  It then follows that we can identify every point $x\in M$ with the unique point $\tilde x\in D$ having the property
that $\kappa(\tilde x)=x$.  Let also $\tilde \gamma(t)$, $|t|\le \tfrac12$ similarly denote those points in $D$ corresponding to our geodesic segment
$\gamma(t)$, $|t|\le \tfrac12$ in $M$.  Then $\{\tilde \gamma(t): |t|\le \tfrac12\}$ is a line segment of unit length whose midpoint is the origin, and we shall denote just by
$\tilde \gamma$ the line through the origin containing this segment.  Note that $\tilde \gamma$ then is a geodesic in $\Rt$ for the metric $\tg$, and 
the Riemannian distance between two points on $\tilde \gamma$ agrees with their Euclidean distance.
Finally, if $\Delta_{\tg}$ denotes the Laplace-Beltrami operator associated to $\tg$ then since solutions of the Cauchy problem \eqref{2.7} correspond exactly
to periodic (i.e. $\Gamma$-invariant) solutions of the corresponding Cauchy problem associated to $\partial^2_t-\Delta_{\tg}$, we have the following
important formula relating the wave kernel on $(M,g)$ to the one for the universal cover $(\Rt,\tg)$:
\begin{equation}\label{2.8}
\bigl(\cos \tau \sqrtd\big)(x,y)=\sum_{\alpha\in \Gamma}\bigl(\cos \tau \sqrtg\bigr)(\tilde x, \alpha(\tilde y)).
\end{equation}

The simple geometric facts that we require is in the following variation of \cite[Lemma~3.2]{ChSo}:

\begin{lemma}\label{geomlemma}  Let $\tilde \gamma_1(t)$ and $\tilde \gamma_2(s)$ be two distinct geodesics in $({\mathbb R}^2, \tilde g)$ each parameterized by arc length.  Put
$$\phi(t,s)=d_{\tilde g}(\tilde \gamma_1(t),\tilde \gamma_2(s)).$$
Then if there is a point $(t_0,s_0)\in {\mathbb R}\times {\mathbb R}$ such that $\partial_t\phi(t_0,s_0)=\partial_s\phi(t_0,s_0)=0$
and $\gamma_1(t_0)\ne \gamma_2(s_0)$, then
\begin{equation}\label{g.1}
|\partial_s\phi(t,s)|+|\partial_t\phi(t,s)|\ne 0 \quad \text{if } \, \tilde \gamma_1(t)\ne \tilde \gamma_2(s) \, \, \, \text{and } \, (t,s)\ne (t_0,s_0),
\end{equation}
%Furthermore, at such a point we have
and
\begin{equation}\label{g.2}
\partial_t\partial_s\phi(t_0,s_0)\ne 0.
\end{equation}
\end{lemma}

\begin{proof}  To prove \eqref{g.1} we first note that if  $\tilde \gamma_1(t_0)\ne \tilde \gamma_2(s_0)$ then
$|\partial_t \phi(t_0,s_0)|+|\partial_s\phi(t_0,s_0)|=0$ if and only if the geodesic connecting the points $\tilde \gamma_1(t_0)$ and $\tilde \gamma_2(s_0)$
is perpendicular to both $\tilde \gamma_1$ and $\tilde \gamma_2$ at the unique intersection points.  Since $({\mathbb R}^2,\tilde g)$ has negative
curvature there cannot be another point $(t_1,s_1)\in {\mathbb R}\times {\mathbb R}$ with this property.  For if $t_0\ne t_1$ and $s_0\ne s_1$ the 
geodesic quadrilateral with vertices $\tilde \gamma_1(t_0), \tilde \gamma_1(t_1), \tilde \gamma_2(s_0)$ and $\tilde \gamma_2(s_1)$ 
would have
total angle $2\pi$, which is impossible due to the fact that $({\mathbb R}^2, \tilde g)$ is negatively curved, and similarly if $t_0=t_1$ but $s_0\ne s_1$ then the geodesic triangle with vertices $\tilde \gamma_1(t_0), \tilde \gamma_2(s_0)$ and $\tilde \gamma_2(s_1)$ would have total angle of more than
$\pi$, which is also impossible.

To prove \eqref{g.2} we may assume that $t_0=0$ and work in geodesic normal coordinates vanishing at $\gamma_1(0)$ so that $\gamma_1$
is the $x_1$-axis, i.e., $\gamma_1(t)=(t,0)$.  Then if $\gamma_2(s)=(x_1(s),x_2(s))\ne (0,0)$,
$$\frac{\partial \phi}{\partial t}(0,s)=\frac{-x_1(s)}{\sqrt{x^2_1(s)+x^2_2(s)}}.$$
Our assumption that $\frac{\partial \phi}{\partial t}(0,s_0)=0$ and $\gamma_1(0)\ne \gamma_2(s_0)$ means that $x_1(s_0)=0$ and $x_2(s_0)\ne 0$.
In our coordinates the geodesic connecting $(0,0)=\gamma_1(0)$ and $\gamma_2(s_0)$ is the $x_2$-axis, and for it to be orthogonal to $\gamma_2$
at $\gamma_2(s_0)$, we must have $x_2'(s_0)=0$ and so $x'_1(s_0)\ne 0$.  But then
$$\frac{\partial^2 \phi}{\partial s\partial t}(0,s_0)=-\frac{x_1'(s_0)}{|x_2(s_0)|}\ne 0,$$
which is \eqref{g.2}.
\end{proof}

We also need the following simple stationary phase lemma.

\begin{lemma}\label{statlemma}  Let $a\in C^\infty_0(\Rt)$ and assume that $\phi\in C^\infty(\Rt)$ is real.
Put
$$I(\la)=\iint e^{i\la \phi(t,s)}a(t,s)\, dt ds, \quad \la\ge 1.$$
Then
\begin{equation}\label{ss.1}
|I(\la)|\le C\la^{-1} \, \, \, \text{if } \, \nabla_{t,s}\phi(t,s)\ne 0, \, \, (t,s)\in \text{supp }a.
\end{equation}
Also, if there is a unique point $(t_0,s_0)\in \text{supp }a$ at which $\nabla_{t,s}\phi(t_0,s_0)=0$ and if $\frac{\partial^2\phi}{\partial t \partial s}(t_0,s_0)\ne 0$
then given any $\e>0$ there is a constant $C_\e$ so that
\begin{equation}\label{ss.2}
|I(\la)|\le \e \la^{-\frac12}+C_\e \la^{-1}.
\end{equation}
\end{lemma}

\begin{proof}  The first assertion, \eqref{ss.1}, just follows via integration by parts.  To prove \eqref{ss.2} we assume that there is a $(t_0,s_0)\in \supp a$
at which $\nabla \phi$ vanishes but $\nabla \phi(t,s)\ne0$, $(t,s)\in \supp a\backslash \{(t_0,s_0)\}$ and $\partial_t\partial_s \phi(t_0,s_0)\ne 0$.
We can split matters into two further cases: (i) $\partial^2_t\phi(t_0,s_0)=0$ and (ii) $\partial^2_t\phi(t_0,s_0)\ne 0$.  

In case (i), we note that our assumptions mean that at $(t_0,s_0)$ the mixed Hessian of $\phi$ satisfies
$$\text{det } \left( \begin{array}{cc}
\phi''_{tt}&\phi''_{ts}
\\
\phi''_{ts}&\phi''_{ss}
\end{array}\right) \ne0,
$$
and, therefore, by two-dimensional stationary phase %(see e.g.  \cite[Chapter 1]{S3}), 
we have $|I(\la)|\le C\la^{-1}$.

To finish, it suffices to show that we have \eqref{ss.2} under the assumption that
$$\nabla \phi(t,s)\ne 0, \, \, \, (t,s)\in \supp a\backslash \{(t_0,s_0)\}, \, \, \text{and } \, \partial_t^2\phi(t_0,s_0)\ne 0.$$
If we let $\beta$ be as above, it follows from \eqref{ss.1} that given any fixed $\e>0$, we have
\begin{equation}\label{ss.3}
I_1(\la)=\iint e^{i\la\phi(t,s)}\bigl(1-\beta(\e^{-1}(t-t_0))\beta(\e^{-1}(s-s_0))\bigr) \, a(t,s)\, dt ds \, = \, O_\e(\la^{-1}).
\end{equation}
Furthermore, if $\e>0$ is chosen so that $\partial_t^2\phi\ne 0$ on $\text{supp }\beta(\e^{-1}(\, \cdot\, -t_0))\beta(\e^{-1}(\, \cdot\, -s_0))$,
it follows from one-dimensional stationary phase that for each fixed $0\le s\le1$ we have
$$\Bigl| \beta(\e^{-1}(s-s_0))\int_{-\infty}^\infty e^{i\la\phi(t,s)}\beta(\e^{-1}(t-t_0)) \, a(t,s)\, dt\Bigr|\le C\la^{-\frac12},$$
where $C$ is independent of $\e>0$ (cf. the proof of \cite[Theorem 1.1.1]{S3}).  This clearly implies that 
we have the uniform bounds $|I(\la)-I_1(\la)|\le A\e \la^{-\frac12}$
for all small $\e>0$.  By combining this inequality with \eqref{ss.3} we deduce
that \eqref{ss.2} holds in this case as well, which finishes the proof.
\end{proof}

%\begin{lemma}\label{statlemma}  Let $\phi(t,s)\in C^\infty(\R\times \R)$ and $\varphi(t)\in C^\infty(\R)$ be real.  Then if $a(t,s)\in C^\infty(\R\times \R)$,
%$\delta>0$ and
%\begin{equation}\label{s.1}
%|\partial_t\phi(t,s)|+|\partial_s\phi(t,s)|\ge \delta, \quad (t,s)\in \text{supp }a\cap [0,1]\times [0,1],
%\end{equation}
%then
%\begin{equation}\label{s.2}
%\Bigl| \int_0^1\int_0^1 e^{i\la \phi(t,s)}a(t,s)\, dt ds\Bigr|\le C\la^{-1},
%\end{equation}
%where $C=C(\delta, \phi,a)$ depends on $\delta$ and the size of finitely many derivatives of $\phi$ and $a$ on $[0,1]\times [0,1]$.  Similarly,
%if $b(t)\in C^\infty(\R)$ and
%\begin{equation}\label{s.3}
%|\partial_t\varphi(t)|+|\partial^2_t\varphi(t)|\ge \delta, \quad 0\le t\le 1,
%\end{equation}
%then
%\begin{equation}\label{s.4}
%\Bigl|\int_0^1 e^{i\la \varphi(t)} b(t)\, dt\Bigr|\le C\la^{-\frac12},
%\end{equation}
%where $C=C(\delta,\varphi,b)$ depends on $\delta$ and the size of finitely many derivatives of $\varphi$ and $b$ in $[0,1]$.
%\end{lemma}

To use Lemmas~\ref{geomlemma}-\ref{statlemma}, we also require another result which is essentially Lemma 3.1 in \cite{ChSo}.

\begin{lemma}\label{lemma2.4}
Given $\alpha\in \Gamma$ set
\begin{multline}\label{3.11}
K^\gamma_{T,\lambda,\alpha}(t,s)
\\
=\frac1{\pi }\int_{-\infty}^\infty (1-\beta(\tau)) \, \Hat \chi(\tau/T) e^{-i\la\tau} 
\bigl(\cos \tau \sqdt\bigr)(\tilde \gamma(t), \alpha(\tilde \gamma(s)))\, d\tau, \, \, |t|, |s|\le 1/2.
\end{multline}
Then if $\alpha\ne Identity$ and we set
\begin{equation}\label{3.13}
\phi_{\gamma,\alpha}(t,s)=d_{\tilde g}(\tilde \gamma(t),\alpha(\tilde \gamma(s))), \quad |t|, |s|\le 1/2,
\end{equation}
then we can write for $|t|, |s|\le 1/2$
\begin{multline}\label{3.14}
K^\gamma_{T,\la,\alpha}(t,s)=
%\\
\la^{\frac12}
w(\tilde \gamma(t),\alpha(\tilde \gamma(s)))
\sum_\pm a_\pm(T,\la; \,  \phi_{\gamma,\alpha}(t,s))e^{\pm i\la \phi_{\gamma,\alpha}(t,s)}
\\
+R^\gamma_{T,\la,\alpha}(t,s),
\end{multline}
where $w(x,y)$ is a smooth bounded function on $\Rtwo\times\Rtwo$ and where for each $j=0,1,2,\dots$ there is a constant
$C_j$ independent of $T,\la\ge1$ so that
\begin{equation}\label{3.15}
|\partial_r^ja_\pm(T,\la; \, r)|\le C_j r^{-\frac12-j}, \quad r\ge 1,
\end{equation}
and for a constant $C_T$ which is independent of $\gamma$, $\alpha$ and $\la$
\begin{equation}\label{3.16}
|R_{T,\la,\alpha}^\gamma(t,s)|\le C_T\la^{-1} 
%\quad \text{if } \, K^\gamma_{T,\la,\alpha}\not\equiv 0
.
\end{equation}
%Additionally,  for each $j=0,1,2,\dots$ 
%and $k=0,1,2,\dots$ we have
% \begin{equation}\label{3.18}
%  |\partial_t^j\partial_s^k \phi_{\gamma,\alpha}(t,s)|
% \le C_{T,j,k}, \quad \text{if } \, \alpha \ne Identity, \, \text{and } \, K^\gamma_{T,\la,\alpha}\not\equiv0.
% \end{equation}
\end{lemma}

We shall postpone the proof of this result until the end and use it now, along with Lemmas~\ref{geomlemma} and \ref{statlemma}, to prove \eqref{2.6}.
Recall that \begin{multline}\label{2.21}
\bigl(\cos \tau\sqdt \bigr)(\tilde x,\tilde y)=0 \quad \text{if } \, \, d_{\tilde g}(\tilde x,\tilde y)>|\tau|, \\ \text{and } \, \bigl(\cos \tau\sqdt \bigr)(\tilde x,\tilde y)
\, \, \text{is smooth if  } \, \, \dgt(x,y)\ne |\tau|.
\end{multline}
Therefore since $\beta(\tau)=1$ for $|\tau|\le 3$ we have $(1-\beta(\tau))(\cos \tau \sqrtg)(\tilde \gamma(t),\tilde \gamma(s))\in C^\infty$ if $0\le s,t\le1$,
and so
$$\int_{-T/2}^{T/2} (1-\beta(\tau))\, \Hat  \chi(\tau/T) e^{-i\tau\la} \bigl(\cos \tau\sqrtg \bigr)(\tilde \gamma(t),\tilde\gamma(s))
\, d\tau =O_T(\la^{-N}),
$$
for any $N=1,2,3,\dots$.

As a result, if we use \eqref{2.8}, we conclude that we would have \eqref{2.6} if we could show that
\begin{multline}\label{2.22}
\sum_{\alpha \in \Gamma \backslash Identity}
\Bigl|\iint\int_{-T/2}^{T/2} b(t,s)(1-\beta(\tau))\Hat  \chi(\tau/T) e^{-i\tau\la} \bigl(\cos \tau\sqrtg \bigr)(\tilde \gamma(t),\alpha (\tilde\gamma(s)))
\, d\tau dt ds\Bigr|
\\
\le 1+C_T\la^{-\frac12}.
\end{multline}
By \eqref{2.21} there are are only finitely many nonzero summands here (actually $O(\exp(cT))$ ones).  Consequently, we would have
\eqref{2.22} if we could show that given $\e>0$ and $Identity \ne \alpha \in \Gamma$ we have
\begin{multline}\label{2.23}
\Bigl|\iint\int_{-T/2}^{T/2}b(t,s) (1-\beta(\tau))\Hat  \chi(\tau/T) e^{-i\tau\la} \bigl(\cos \tau\sqrtg \bigr)(\tilde \gamma(t),\alpha (\tilde\gamma(s)))
\, d\tau dt ds\Bigr|
\\
\le \e + C_{\alpha,\e}\la^{-\frac12}.
\end{multline}

Note that since $\tilde \gamma$ is a geodesic in $(\Rt,\tilde \gamma)$ and $\alpha$ is an isometry, it follows that $\alpha(\tilde \gamma)$ is also a geodesic.  It is a geodesic
which is different from $\tilde \gamma$ if $\alpha$ is not in the stabilizer subgroup of $\Gamma$ of all deck transformations preserving $\tilde \gamma$.  If our geodesic $\gamma$ in $(M,g)$ is not a periodic geodesic then the stabilizer subgroup is just the identity element.  If $\gamma$ is a periodic geodesic in $M$ of minimal period $\ell>0$, then we must have $\ell \ge 10$, because of our assumption regarding the injectivity radius.  In this case, every nontrivial element of the stabilizer group satisfies $\alpha(\tilde \gamma(s))=\tilde \gamma(s+k\ell)$ for some $k\in {\mathbb Z}\backslash 0$.  By Lemma~\ref{lemma2.4} for such
a $\alpha$ with $k\ne 0$, modulo a term which is $O_\alpha(\la^{-1})$ we have that the left side of \eqref{2.23} is equal to the sum over $\pm$ of
$$\la^{\frac12}\Bigl|\iint b(t,s) \, w(\tilde \gamma(t), \tilde \gamma(s+k\ell))\,  a_\pm\bigl(T,\la; |t-s-k\ell |)\, e^{\pm i\la |t-s-k\ell |} \, dt ds\Bigr|,
$$
with $a_\pm$ as in \eqref{3.15}.  Since $b\in C^\infty_0(\Rt)$ vanishes when $|t|$ or $|s|$ is larger
than $\frac12$ and $w\in C^\infty(\Rt)$, by a simple integration by parts argument this term is $O_\alpha(\la^{-\frac12})$, which means that we have
\eqref{2.23} for all nontrivial elements of the stabilizer group of $\tilde \gamma$.

To prove that we also have \eqref{2.23} for the remaining case where $\alpha\in \Gamma$ is not in the stabilizer group, by the above, it is enough to show that, if
$\phi_{\gamma,\alpha}(t,s)$ is as in \eqref{3.13}, then given $\e>0$ there is a constant $C_{\alpha,\e}$ so that
\begin{multline}\label{last}\Bigl|\iint b(t,s) \, w(\tilde \gamma(t), \alpha(\tilde \gamma(s))) \, 
a_\pm(T,\la; \phi_{\gamma,\alpha}(t,s)) \, e^{\pm i\la \phi_{\gamma,\alpha}(t,s)} \, dt ds\Bigr|
\\
\le \e\la^{-\frac12}+C_{\alpha,\e}\la^{-1}.\end{multline}
By Lemma~\ref{geomlemma} the phase function here either satisfies $\nabla_{t,s}\phi_{\gamma,\alpha}\ne0$ on the support of the
integrand, or there is a unique point $(t_0,s_0)$ in the support where $\nabla_{t,s}\phi_{\gamma,\alpha}$ vanishes and at that point 
$\frac{\partial \phi_{\gamma,\alpha}}{\partial t\partial s}(t_0,s_0)\ne 0$.  In the former case by \eqref{ss.1} the left
side of \eqref{last}  $O_{\alpha}(\la^{-1})$, which is more than is required.  In the latter case, we obtain
the estimate from Lemma~\ref{statlemma}.  Thus, we have established \eqref{last}, which except for the proof of Lemma~\ref{lemma2.4}, completes the proof of Theorem~\ref{theoremneg}.

\begin{proof}[Proof of Lemma~\ref{lemma2.4}]
Since the injectivity radius of $(M,g)$ is 10 or more, it follows that $\dgt(\tilde \gamma(t),\alpha(\tilde \gamma(s)))\ge 10$ 
if $|t|, |s|\le 1/2$.  Also, as noted before, for each $T$, by Huygens principle, there are only finitely many terms in \eqref{3.11}
that we must consider.
%and hence
%\eqref{3.18} is valid since the distance function is smooth away from the diagonal.  Also since $(\cos \tau\sqrtg )(x,y)$ is smooth
%when $\dgt (x,y)<|\tau|$ and $\beta(\tau)=1$ for $|\tau|\le 3$, we need only show that the other assertions in the lemma are valid
%when $\pg(t,s)\ge 1$.
%

Next, let for $x\in \Rt$, $|x|\ge1$,
\begin{align*}
K_0(|x|)&=\frac1\pi \int_{\Rt}\int_{-\infty}^\infty \Hat \chi(\tau/T)e^{i\la\tau}\cos(\tau|\xi|) e^{ix\cdot \xi} \, d\tau d\xi
\\
&=\int_{\Rt} T\chi(T(\la-|\xi|)) e^{ix\cdot \xi} \, d\xi \, + \, \int_{\Rt}T\chi(T(\la+|\xi|)) e^{ix\cdot \xi}\, d\xi.
\end{align*}
Also let $\Phi_T(\xi)\in {\mathcal S}(\R)$ be defined by the Fourier transform $\Hat \Phi_T(\tau)=\beta(\tau)\Hat \chi(\tau/T)$ and put
\begin{align*}
K_1(|x|)&=\frac1\pi \int_{\Rt}\int_{-\infty}^\infty \beta(\tau) \Hat\chi(\tau/T)e^{i\la \tau}\cos(\tau|\xi|) \, d\tau d\xi
\\
&=\int_{\Rt}\Phi_T(\la-|\xi|) e^{ix\cdot \xi} \, d\xi \, + \, \int_{\Rt}\Phi_T(\la+|\xi|)e^{ix\cdot \xi} \, d\xi.
\end{align*}
Recall that the Fourier transform of Lebesgue measure on the circle is of the form
$$\widehat{d\theta}(y)=\int_{S^1}e^{iy\cdot (\cos\theta,\sin\theta)} \, d\theta
=|y|^{-\frac12}\sum_\pm a_\pm(|y|) e^{\pm i |y|}, \quad |y|\ge 1,$$
where for every $j=0,1,2,\dots$ we have
$$|\partial^j_ra_\pm(r)|\le C_jr^{-j}, \quad r\ge1.$$
Also $\widehat{d\theta}\in C^\infty(\R)$.  Therefore, if $\la, T\ge1$, modulo a term which is $O((\la|\xi|)^{-N})$ for any $N$ independent of $T$, we have
\begin{align}\label{2.26}
K_0(|x|)&=|x|^{-\frac12}\int_0^\infty T\chi(T(\la-r)) \, \sum_\pm a_\pm (|x|r)e^{\pm ir|x|} \, r^{\frac12}dr
\\
&=\la^{\frac12}|x|^{-\frac12}\sum_\pm b_\pm (T,\la; |x|)e^{\pm i\la|x|},
\notag
\end{align}
where one easily sees that
\begin{equation}\label{2.27}
|\partial^j_r b_\pm(T,\la;r)|\le C_jr^{-j}, \quad r,\la, T\ge 1, \, \, j=0,1,2\dots .
\end{equation}
Similar arguments show that, modulo an $O((\la|x|)^{-N})$ error we also have
\begin{equation}\label{2.28}
K_1(|x|)=\la^{\frac12}|x|^{-\frac12}\sum_\pm \tilde b_\pm(T,\la; |x|)e^{\pm i\la|x|},
\end{equation}
where $\tilde b_\pm$ satisfy the bounds in \eqref{2.27}.

To use this we shall use the Hadamard parametrix (see \cite{SoggeHang}).  For $x,y\in \Rt$ we can write
\begin{align}\label{2.29}
\bigl(\cos \tau \sqrtg\bigr)(x,y)&=(2\pi)^{-2}w(x,y)\int_{\Rt}e^{i\dgt(x,y)\xi_1}\cos(\tau|\xi|)\, d\xi
\\
&+\sum_\pm \int_{\Rt}e^{i\dgt(x,y)\xi_1}e^{\pm i\tau|\xi|}a_\pm(\tau,x,y,|\xi|)\, d\xi +R(\tau,x,y),
\notag
\end{align}
where we can take the remainder to satisfy
$$|R(\tau,x,y)|+|\partial_\tau R(\tau,x,y)|\le C_T, \quad \text{if } \, |\tau|\le T,$$
and $a_\pm$ is a symbol of order $-2$ which, in particular, satisfies
\begin{equation}\label{2.30}
|\partial_\tau^j a_\pm(\tau,x,y,|\xi|)|\le C_{T,j}(1+|\xi|)^{-2}, \quad \text{if } \, \, |\tau|\le 2T, \, \, j=0,1,2,\dots ,
\end{equation}
and where the leading coefficient $w$ is smooth, nonnegative and satisfies
$$w(x,y)\le C$$
independent of $x,y\in \Rt$, by volume comparison theorems (see \cite{Berard}, \cite{SZ4}).

Clearly if we replace $(\cos \tau \sqrtg)(\tilde \gamma(t),\alpha(\tilde \gamma(s)))$ by $R(\tau,\tilde \gamma(t),\alpha(\tilde \gamma(s)))$ in
\eqref{3.11} we can integrate by parts in $\tau$ to see that the resulting expression satisfies the bounds in \eqref{3.16}.

If we take $x=\tilde \gamma(t)$ and $y=\alpha(\tilde \gamma(s))$, $|t|, |s|\le 1/2$ for the first term in the right side of \eqref{2.29} and replace the 
cosine-transform kernel in \eqref{3.11} by this expression then we will exactly obtain $(2\pi)^2$ times
$$w(\tilde \gamma(t),\alpha(\tilde \gamma(s))) K_0(\pg(t,s))-w(\tilde \gamma(t),\alpha(\tilde \gamma(s))) K_1(\pg(t,s)),
$$
which can be taken as the first term in the right side of \eqref{3.14}, with \eqref{3.15} being valid.  Finally since \eqref{2.30} holds, the proof
of this last assertion also shows that if in \eqref{3.11} we replace the cosine-transform kernel by the second term in the right side of
\eqref{2.29} we obtain another term satisfying the bounds in \eqref{3.16} (in fact it is $O_T(\la^{-\frac32})$), which completes 
the proof.
\end{proof}

%\begin{remark}  If $\{\gamma(s): \, 0\le s\le L\}$ is a geodesic parameterized by arc length on a compact negatively curved Riemannian surface, the
%then proof of Theorem~\ref{theoremneg} shows that the measures $e_j(\gamma(s))\, ds$ on $[0,L]$ tend to zero in the sense of distributions, by
%which we mean that if $u\in C^\infty([0,L])$ is fixed, we have
%$$\int_0^Tu(s)\, e_j(\gamma(s))\,  ds \to 0, \quad \text{as } \, \, j\to \infty.$$
%\end{remark}

\end{document}